\documentclass{amsart}
\usepackage{mathrsfs}
\usepackage[english]{babel}
\usepackage{enumerate}
\usepackage{xfrac}
\usepackage{mathtools}
\usepackage{amsmath}

\usepackage[T1]{fontenc}

\theoremstyle{plain}
\newtheorem*{maintheo*}{Main Theorem}

\newtheorem{theorem}{Theorem}[section]
\newtheorem{proposition}[theorem]{Proposition}
\newtheorem{corollary}[theorem]{Corollary}
\newtheorem{lemma}[theorem]{Lemma}

\theoremstyle{definition}

\newtheorem{defandrem}[theorem]{Definition and Remarks}

\newtheorem{example}[theorem]{Example}
\newtheorem{remark}[theorem]{Remark}
\newtheorem{notandrem}[theorem]{Notation and Remarks}

\theoremstyle{remark}

\numberwithin{equation}{section}

\newcommand{\N}{\mathbb N}
\newcommand{\Z}{\mathbb Z}

\newcommand{\R}{\mathbb R}

\newcommand{\Id}{\mathrm{Id}}
\newcommand{\GL}{\operatorname{GL}}

\newcommand{\Ot}{\operatorname{O}}

\newcommand{\Sp}{\operatorname{Sp}}
\newcommand{\tSp}{\widetilde{\operatorname{Sp}}}

\newcommand{\Aut}{\operatorname{Aut}}

\newcommand{\SP}{\mathcal{P}}

\newcommand{\SPs}{\mathcal{P}^*}

\newcommand{\diag}{\operatorname{diag}}

\newcommand{\ba}{\setminus}

\newcommand{\z}{\mathfrak{z}}
\newcommand{\h}{\mathfrak{h}}

\newcommand{\m}{\mathbf{m}}

\newcommand*{\ltrans}[1]{\prescript{\mathrm t}{}{#1}}

\title{Relatively Compact Sets of Heisenberg Manifolds}
\author{Sebastian Boldt}
\email{boldt@math.uni-leipzig.de}
\address{Mathematisches Institut\\Universit\"at Leipzig\\04081 Leipzig\\Germany}


\begin{document}
	
\maketitle
\begin{abstract}
	We give a necessary and sufficient condition for a set of left invariant metrics on a compact Heisenberg manifold to be relatively compact in the corresponding moduli space.
\end{abstract}
	
\section{Introduction}

The classic compactness Theorem of Mahler, also called Mahler's Selection Theorem, states that a set of lattices in $\R^n$ has compact closure if and only if the volume of each lattice is uniformly bounded from above and the length of a shortest nontrivial vector in each lattice is uniformly bounded from below. The topology on the space of lattices will become clear later. The statement of the theorem remains true if we replace lattices by flat tori. Indeed, there is a bijective correspondence between (isometry classes of) lattices and (isometry classes of) flat tori, under which shortest nontrivial vectors correspond to shortest nontrivial closed curves.

In this note we are concerned with an extension of this theorem to compact Riemannian Heisenberg manifolds. A compact Heisenberg manifold is a compact quotient of the Heisenberg group by a discrete subgroup. These quotients can be characterised as the total spaces of (nontrivial) $S^1$-bundles over even-dimensional tori. Given a left invariant metric, they become Riemannian submersions whose base is a flat torus and whose $S^1$-fibres are totally geodesic. One might hope that a set of such metrics has compact closure if the condition of Mahler's theorem applies to the base torus and the fibres. Unfortunately, this is not true. One has to bound an additional invariant.

We will now explain the above in greater detail. Let $L$ be a lattice (of full rank) in $\R^n$. Associated with $L$ is the flat torus $T=L\ba \R^n$. The lattice $L$ can be recovered from $T$ as the group of deck transformations of the universal Riemannian covering $\R^n\to T$. We can describe $L$ by a basis, which we view as a matrix $G\in\GL(n;\R)$. This basis is of course not unique, but every other basis of $L$ can be obtained from $G$ by a change of basis, which corresponds to multiplication (from the right) by an element of $\GL(n;\Z)$. Since we are only interested in the isometry class of $L$, we view every lattice that we obtain from an orthogonal transformation applied to $L$ as identical to $L$. This means that the isometry class of $L$ (resp.\ $T$) is given by the element $\Ot(n;\R)\cdot G\cdot\GL(n;\Z)\in \Ot(n;\R)\ba \GL(n;\R) / \GL(n;\Z)$. Now note that the quotient $\Ot(n;\R)\ba \GL(n;\R)$ is diffeomorphic to the space $\SP_{n}$ of symmetric, positive definite $n\times n$-matrices: $\GL(n;\R)$ acts from the right on $\SP_{n}$ via $\SP_n\times \GL(n;\R)\ni (Y,H)\mapsto Y[H]=\ltrans{H}YH\in\SP_{n}$. This action is smooth and transitive and the isotropy group of the identity $\Id_n\in \SP_n$ is exactly $\Ot(n;\R)$, so that we obtain a diffeomorphism $\Ot(n;\R)\ba \GL(n;\R)\to\SP_n$ given by $\Ot(n;\R)\cdot G\mapsto \ltrans GG$. The matrix $\ltrans GG$ is the Gram matrix of the basis $G$ of $L$. The moduli space of lattices in $\R^n$ resp.\ flat $n$-dimensional tori is thus $\SP_n / \GL(n;\Z)$.

Given $[Y]\in \SP_n / \GL(n;\Z)$, the volume of the corresponding torus $T$ is given by $\det(Y)^{\sfrac 12}$ and the squared length of a shortest nontrivial closed curve by $m(Y)=\inf\{Y[a]=\ltrans aYa\,|\,\in\Z\ba\{0\}\}$. Mahler's compactness Theorem (see Corollary~\ref{cor:Mahler-Selection-Theorem}) states that a set $M\subset \SP_n / \GL(n;\Z)$ has compact closure if and only if there are constants $C_0,C_1>0$ such that $\det(Y)\leq C_1$ and $m(Y)\geq C_0$ for all $[Y]\in M$.

In this note we prove an analogous result for the moduli space of compact Riemannian Heisenberg manifolds. Let $H_n$ be the $(2n+1)$-dimensional Heisenberg group and fix a cocompact discrete subgroup $\Gamma\subset H_n$. A Riemannian metric $\m$ on $\Gamma \ba H_n$ is called left invariant if its lift to $H_n$ is left invariant, i.e., invariant under left multiplication by $H_n$ on itself. The moduli space of such Riemannian metrics is essentially $\SP_{2n}/\tSp(2n;\Z)\times (0,\infty)$ (see Theorem~\ref{thm:moduli-space-normalised-heisenberg-manifolds} for the exact definition). Here, $\tSp(2n;\Z)$ is the space of all matrices $A\in\GL(2n;\R)$ such that $\ltrans AJA=\pm J$, where $J$ is the matrix representation of the standard almost complex structure $J$ on $\R^{2n}$. We know what the compact subsets of $(0,\infty)$ are. Example~\ref{ex:noncompact-set-in-moduli-space} shows that Mahler's Selection Theorem does not hold for $\SP_{2n}/\tSp(2n;\Z)$, i.e., there is a sequence $\{[Y_k]\}_{k\in\N}$ which does not contain a converging subsequence, yet $m(Y_k)=\det(Y_k)=1$ for all $k$.
	
The solution is to place a uniform bound on the (absolut value of the) eigenvalues of $Y^{-1}J$. Accordingly, our main result, Theorem~\ref{thm:compact-sets-in-moduli-space}, states the following.

\begin{maintheo*}
	Let $M\subset \SP_{2n}/\tSp(2n;\Z)$. Assume there exist positive constants $C_0,C_1,C_2$ such that $m(Y)\geq C_0$, $\det(Y)\leq C_1$ and the eigenvalues of $Y^{-1}J$ are bounded in absolute value from above by $C_2$ for all $[Y]\in M$. Then $M$ has compact closure.
\end{maintheo*}

Necessity of the stated condition follows from the continuity of $\det$, $m$ and the eigenvalues of $Y^{-1}J$.

The proof uses the following idea. Since $\tSp(2n;\Z)\subseteq \GL(2n;\Z)$ is a subgroup of infinite index ($n>1$), a fundamental domain for $\SP_{2n}/\tSp(2n;\Z)$ could be obtained by gluing an infinite number of displacements of a fundamental domain for $\SP_{2n}/\GL(2n;\Z)$. Such a domain is Minkowski's fundamental domain, for which the Hermite-Mahler Compactness Theorem (see Theorem~\ref{thm:hermite-mahler}) characterises the compact sets. We then show that the preimage in $\SP_{2n}$ of a set $M\subset \SP_{2n}/\tSp(2n;\Z)$ which satisfies the assumptions of Theorem~\ref{thm:compact-sets-in-moduli-space} is necessarily contained in a finite number of displacements of Minkowski's fundamental domain and satisfies the assumptions of the Hermite-Mahler Theorem and must therefore be compact. The crucial matrix inequality needed in the proof is established in Lemma~\ref{lem:moduli-space-special-inequality}.

In geometric terms this means that a set of left invariant metrics on a fixed quotient $\Gamma\ba H_n$ is precompact iff the volume of the $S^1$-fibres is bounded from above and below, the systole of the base torus is bounded from below, the volume of the base torus is bounded from above, and that all sectional curvatures are bounded from above by a positive constant, see Remark \ref{rem:sectional-curvatures}.

The results in this note have been part of the author's PhD-thesis \cite{Bol18} and were very recently extended to left invariant sub-Riemannian metrics by K.~Tashiro \cite{Tas20}.

This note is organised as follows. In Section~\ref{sec:mflds} we introduce compact Riemannian Heisenberg manifolds and their moduli spaces. Section~\ref{sec:compactsets} contains the main result (Theorem~\ref{thm:compact-sets-in-moduli-space}) and its proof.

\section*{acknowledgements}

The author gratefully acknowledges financial support by the Collaborative Research Center 647 - Space - Time - Matter.

\section{Riemannian Heisenberg Manifolds}\label{sec:mflds}
\noindent
This section describes compact Riemannian Heisenberg manifolds. The exposition follows mostly \cite{zbMATH03967358}.

For $x,y\in\R^n$, $s\in\R$ we let
\begin{align*}
\gamma(x,y,s) := \begin{pmatrix}		1 & \ltrans x & s \\		0 &  \Id_n & y \\		0 & 0 & 1\\	\end{pmatrix},\quad & X(x,y,s) := \begin{pmatrix}		0 & \ltrans x & s \\		0 &  0 & y \\		0 & 0 & 0\\	\end{pmatrix}\,.
\end{align*}
The \textit{$(2n+1)$-dimensional Heisenberg group} $H_n$ is $H_n := \left\{\gamma(x,y,s)\,|\,x,y\in\R^n,s\in\R\right\}$ with the Lie group structure that it inherits as a closed subset of $\GL(n+2;\R)$. Its Lie algebra is $\mathfrak{h}_n = \left\{X(x,y,s)\,|\,x,y\in\R^n,s\in\R\right\}$. The standard basis $\mathfrak B_n$ of $\h_n$ is 
\[\mathfrak B_n := (X_1,\ldots,X_n,Y_1,\ldots,Y_n,Z)\] with 
\begin{equation*}
\begin{aligned}
X_j & := X(e_j,0,0) \text{ for all } 1\leq j\leq n\,,\\
Y_j & := X(0,e_j,0) \text{ for all } 1\leq j\leq n\,,\\
Z & := X(0,0,1)\,,
\end{aligned}
\end{equation*}
where $(e_1,\ldots,e_n)$ is the standard basis of $\R^n$. Multiplication and inversion in $H_n$ obey the rules
\begin{equation}\label{eqn:H_n-multiplication}
\begin{aligned}
\gamma(x,y,s)\cdot\gamma(x',y',s') & = \gamma(x+x',y+y',s+s'+\langle x,y'\rangle)\,,\\
\gamma(x,y,s)^{-1} & = \gamma\left(-x,-y,-s+\langle x,y\rangle\right)\,,
\end{aligned}
\end{equation}
where $\langle\cdot,\cdot\rangle$ denotes the Euclidean inner product. The Lie exponential map \linebreak$\exp:\h_n\to H_n$,
\begin{align*}
\exp X(x,y,s) &= \Id + X(x,y,s)+\frac{1}{2}X(x,y,s)^2\\
&= \gamma(x,y,s+\tfrac12 \langle x,y\rangle),
\end{align*}
is a global diffeomorphism with inverse $\log:H_n\to\h_n$,
\begin{align*}
\log \gamma(x,y,s) = X(x,y,s-\tfrac 12 \langle x,y\rangle)\,.
\end{align*}
From \eqref{eqn:H_n-multiplication} it follows that commutators in $H_n$ and $\h_n$ are given by
\begin{equation}\label{eqn:Heisenberg-commutators}
\begin{aligned}
\left[\gamma(x,y,s),\gamma(x',y',s')\right] & = \gamma\left(0,0,A((x,y),(x',y'))\right)\,,\\
\left[X(x,y,s),X(x',y',s')\right] & = X\left(0,0,A((x,y),(x',y'))\right)\,,
\end{aligned}
\end{equation}
where $A$ is the standard symplectic form on $\R^{2n}$ whose matrix resepresentation w.r.t.~the standard basis is
\begin{align}\label{def:symplectic-J}
J=\begin{pmatrix}
	0 & \Id_n \\
	-\Id_n & 0
\end{pmatrix}\,.
\end{align}
Since $A$ is nondegenerate it follows from \eqref{eqn:Heisenberg-commutators} that the centre $\z = \z_n$ of $\h_n$ is $\z_n = \{X(0,0,s)\,|\,s\in\R\}$. We identify the subspace $\{X(x,y,0)\,|\,(x,y)\in\R^{2n}\}$ with $\R^{2n}$. Under this identification $\h_n$ is the direct sum $\h_n=\R^{2n}\bigoplus\z$ and $(X_1,\ldots,X_n,Y_1,\ldots,Y_n)$ is the standard basis of $\R^{2n}$. By \eqref{eqn:Heisenberg-commutators} we have for all $X,Y\in\R^{2n}$
\begin{equation*}
[X,Y] = A(X,Y)Z\,.
\end{equation*}

We now turn to the automorphisms of $H_n$ and $\h_n$. Since $H_n$ is connected and by the naturality of the Lie exponential map, we will identify automorphisms of $H_n$ with their differentials at the identity. These are precisely the automorphisms of $\h_n$. Furthermore, we will identify an automorphism $\varphi\in\Aut(\h_n)$ with its matrix representation relative to the basis $\mathfrak{B}_n$. Let  \[
\widetilde{\Sp}(2n;\R):=\left\{\beta\in\GL(2n;\R)\,|\,\ltrans \beta J\beta=\epsilon\left(\beta\right)J,\, \varepsilon\left(\beta\right)=\pm1\right\}\,,
\]
 with $J$ as in \eqref{def:symplectic-J}. We imbed $\tSp(2n;\R)$ into $\GL(2n+1;\R)$ via
\begin{align*}
\tSp(2n;\R)\ni\beta\mapsto \begin{pmatrix}	\beta 	& 0 \\	0	& \epsilon\left(\beta\right)\end{pmatrix} \in\GL(2n+1;\R)\,.
\end{align*}
For $a\in\R\ba\{0\}$ and $w\in\R^{2n}$, let
\begin{align*}
\alpha(a,w):=\begin{pmatrix} a\cdot \Id_{2n} & 0 \\	w^t  & a^2 	\end{pmatrix}\in\GL(2n+1;\R)\,.
\end{align*}
Simple calculations show that the group $\Aut\left(\h_n\right)$ consists of all products of the form $\alpha(a,w)\cdot \beta$, $a\in\R\ba\{0\},w\in\R^{2n},\beta\in\widetilde{\Sp}(2n;\R)$. The inner automorphisms are those for which $a=1$ and $\beta=\Id_{2n}$.

We introduce a set of distinguished uniform (i.e., discrete and cocompact) subgroups of $H_n$. Let
\begin{equation*}
\mathcal{D}_n:=\left\{ r=\left(r_1,\ldots,r_n\right)\in \N^n\,|\, \forall i\in\{1,\ldots,n-1\}:\,r_i|r_{i+1}\right\}\,.
\end{equation*} 
For an $r\in\mathcal{D}_n$ define the matrix 
\begin{align*}
\delta_r:=\diag\left(r_1,\ldots,r_n,1,\ldots,1\right)
\end{align*}
and the uniform subgroup $\Gamma^r\subset H_n$ by
\begin{equation*}
\Gamma^r := \left\{\gamma(x,y,s)\,|\,\left(\begin{smallmatrix}
x\\ y
\end{smallmatrix}\right)\in\delta_r\cdot\Z^{2n},s\in\Z \right\}\,.
\end{equation*}

We now address the Riemannian structure of Heisenberg manifolds. We call any Riemannian metric $\m$ on a compact Heisenberg manifold $\Gamma\ba H_n$ that originates from a left invariant metric on $H_n$ \textit{left invariant}. Any such metric is uniquely determined by its values at the (equivalence class of the) identity, i.e., by the induced inner product on $\h_n$. Via the basis $\mathfrak{B}_n$, we identify the space $\mathscr M(\h_n)$ of inner products on $\h_n$ with $\SP_{2n+1}$, the space of symmetric positive definite $(2n+1)\times (2n+1)$-matrices. In case that for a given metric $\m\in\SP_{2n+1}$ the subspaces $\z\subset\h_n$ and $\R^{2n}\subset\h_n$ are orthogonal, the metric takes the form
\begin{equation*}
\m = \begin{pmatrix}	h & 0 \\	0 & g	\end{pmatrix} \in\SP_{2n}\times (0,\infty)\,.
\end{equation*}
We call such a metric \textit{normalised} and write $\m=(h,g)$. We also call a compact Riemannian Heisenberg manifold $\left(\Gamma^r\ba H_n,\m\right)$, where $\m=(h,g)$ is a normalised metric, a \textit{normalised Heisenberg manifold}.

If $\m=(h,g)$ is a normalised metric and $\alpha(a,w)\cdot\beta\in\Aut(\h_n)$, then $(\alpha(a,w)\cdot\beta)^*\m$ is a normalised metric if and only if $w=0$. Note that w.r.t.~our identifications, the pullback of a metric $\m$ by an automorhism $\varphi$ corresponds to (the restriction of) the action of $\GL(2n+1;\R)$ on $\SP_{2n+1}$: $\varphi^*\m=\m[\varphi]=\ltrans{\varphi\m\varphi}$.

Furthermore, for $r\in\mathcal D_n$ we have $\alpha(a,0)\cdot\beta \log\Gamma^r=\log\Gamma^r$ if and only if $a=1$ and $\beta\in\delta_r\GL(2n;\Z)\delta_r^{-1}$. We therefore define
\begin{align*}
G_r:= \delta_r\GL(2n;\Z)\delta_r^{-1}\,, & \qquad & \Pi_r:= G_r\cap \widetilde{\Sp}(2n;\R)\,.
\end{align*}

\begin{theorem}[{\cite[Corollary~2.5, Theorem~2.7]{zbMATH03967358}}]\label{thm:moduli-space-normalised-heisenberg-manifolds}
	Every compact Riemannian Heisenberg manifold is isometric to a normalised Heisenberg manifold. Moreover, two normalised Heisenberg manifolds $\left(\Gamma^r\ba H_n,\m\right)$ and $\left(\Gamma^{s}\ba H_n,\m'\right)$ are isometric iff $r=s$ and $\m' = \beta^* \m$ for some $\beta\in \Pi_r$. They are homeomorphic iff $r=s$. Accordingly, the set
	\begin{align*}
	\mathcal{M}_n := \bigcup_{r\in\mathcal{D}_n} \mathcal{M}_{n}^r\,, &\qquad \mathcal{M}_{n}^r := \left(\SP_{2n}\times(0,\infty)\right)/\,\Pi_r
	\end{align*}
	parametrises the isometry classes of compact Riemannian Heisenberg manifolds.
\end{theorem} 

The \textit{Kaplan map} $j:\z\to\mathfrak{so}(\R^{2n},h)$ of a normalised Heisenberg manifold \linebreak$(\Gamma^r\ba H_n,\m)$, $\m=(h,g)$, is defined by
\begin{align}\label{eqn:definition-Kaplan-j}
\langle j(W)X,Y\rangle_\m = \langle W,[X,Y]\rangle_\m\,,
\end{align}
for all $W\in\z, X,Y\in\R^{2n}$.
Here, $\mathfrak{so}(\R^{2n},h)$ denotes the space of skew-symmetric endomorphisms of $(\R^{2n},h)$ which we identify via the standard basis with the $2n\times 2n$-matrices which are skew-symmetric w.r.t.~$h$. From \eqref{eqn:definition-Kaplan-j} it follows that \begin{align}\label{eqn:Kaplan-j-matrix}
	j(g^{-\sfrac 12}Z)=-g^{\sfrac 12}h^{-1}J\,.
\end{align}The normalised Heisenberg manifold $(\Gamma^r\ba H_n,\m)$ is of \textit{Heisenberg type} if the eigenvalues of $j(g^{-\sfrac 12}Z)$ are $\pm i$, i.e., if the eigenvalues of $h^{-1}J$ are $\pm ig^{-\sfrac 12}$. 

At last, we mention the sectional curvatures of $(\Gamma^r\ba H_n,\m)$ which can be easily expressed using the Kaplan map. Let $X,Y\in\R^{2n}$ be orthonormal elements of $(\h_n,\m)$. Then the sectional curvature $K$ of $(\Gamma^r\ba H_n,\m)$ is given by
\begin{gather}\label{eqn:sec-curvatures}
	\begin{aligned}
	K(X,Y) &= -\frac 34|[X,Y]|_\m^2\,,\\
	K(X,Z) &=\frac 14 |j(g^{-\sfrac 12}Z)X|_\m^2\,,
	\end{aligned}
\end{gather}
 see \cite[(2.4)]{Ebe94} (note the wrong sign on the right hand side of (2.4) b) though).

\section{Relatively Compact Sets of Riemannian Heisenberg Manifolds}\label{sec:compactsets}

Let $Y\in\SP_{2n}$. From \eqref{eqn:definition-Kaplan-j} and \eqref{eqn:Kaplan-j-matrix} we know that the map $Y^{-1}J$ is skew-symmetric w.r.t.\ the inner product on $\R^{2n}$ defined by $Y$. Thus the eigenvalues of $Y^{-1}J$ are purely imaginary and come in complex conjugate pairs.

\begin{defandrem}\label{def:d_j-h-and-remarks}\leavevmode
	\begin{enumerate}[(i)]
		\item For every $n\in\N$ we define $n$ functions $d_1,\ldots,d_n:\SP_{2n}\to (0,\infty)$ such that $\pm id_k(Y)$, $k=1,\ldots,n$, are the eigenvalues of $Y^{-1}J$ and $d_1\leq d_2\leq \ldots \leq d_n$.
		\item\label{def:d_j-h-and-remarks:enum:continuity} Note that the $d_k$ are continuous (see, e.g., \cite[Theorem~1]{zbMATH03214567}).
		\item\label{def:d_j-h-and-remarks:enum:invariance-under-sp} The functions $d_1,\ldots, d_n:\SP_{2n}\to (0,\infty)$ are invariant under the action of $\widetilde{\Sp}(2n;\R)$. Indeed, for $Y\in\SP_{2n}$ and $A\in\widetilde{Sp}(2n;\R)$ one has
		\begin{align*}
			(Y[A])^{-1}J &= (\ltrans A Y A)^{-1}J = A^{-1}Y^{-1}\ltrans A^{-1}J \sim A (A^{-1}Y^{-1}\ltrans A^{-1}J)A^{-1}\\
			& = Y^{-1}\ltrans A^{-1} J A^{-1} = \varepsilon(A)Y^{-1}J
		\end{align*}
		with $\varepsilon(A)=\pm 1$. By definition of the $d_k$ we thus have $d_k(Y[A])=d_k(Y)$ for all $1\leq k\leq n$, $Y\in\SP_{2n}$ and $A\in\widetilde{\Sp}(2n;\R)$.
		\item Define $\SPs_{2n}(Y) 	:= \left\{Y[S]\,|\,S\in\widetilde{\Sp}(2n;\R)\right\}=\left\{\ltrans SYS\,|\,S\in\widetilde{\Sp}(2n;\R) \right\}$. 
	\end{enumerate}
\end{defandrem}

\begin{proposition}\label{prop:same-d_n-symplectic-leaf}
	Let $X,Y\in\SP_{2n}$. Then
	\[
	X\in\SPs_{2n}(Y) \quad\text{iff}\quad Y\in\SPs_{2n}(X)\quad\text{iff}\quad d_j(X)=d_j(Y) \text{ for all } 1\leq j\leq n\,.
	\]
\end{proposition}
\begin{proof}
	The first 'iff' is a direct consequence of the definition of $\SPs_{2n}(M)$, $M\in\SP_{2n}$.
	
	We first prove the second 'if': By assumption $X^{-1}J\sim Y^{-1}J$. Since $X^{-1}J\sim X^{-\sfrac 12}JX^{-\sfrac 12}$, we have $X^{-\sfrac 12}JX^{-\sfrac 12}\sim Y^{-\sfrac 12}JY^{-\sfrac 12}$, which means that there is $A\in\Ot(2n;\R)$ such that
	\[
	\ltrans AX^{-\sfrac 12}JX^{-\sfrac 12}A=Y^{-\sfrac 12}JY^{-\sfrac 12}\,.
	\]
	This implies that $X^{-\sfrac 12}AY^{\sfrac 12}=:S\in\widetilde{\Sp}(2n;\R)$. By definition of $S$ we have $AY^{\sfrac 12}=X^{\sfrac 12}S$ and hence $Y=Y^{\sfrac 12}\ltrans AAY^{\sfrac 12}=\Id[AY^{\sfrac 12}]=\Id[X^{\sfrac 12}S]=\ltrans SXS$ as claimed. The 'only if' part is proved as follows: If $Y\in\SPs_{2n}(X)$ then $Y=\ltrans S X S$ for some $S\in\widetilde{\Sp}(2n;\R)$. But then $Y^{-1}J=S^{-1}X^{-1}\ltrans S^{-1}J=\pm S^{-1}X^{-1}JS\sim \pm X^{-1}J$ since $S\in\widetilde{\Sp}(2n;\R)$.
\end{proof}

Theorem \ref{thm:moduli-space-normalised-heisenberg-manifolds} now takes the following form for compact Heisenberg manifolds of Heisenberg type.
\begin{corollary}\label{cor:Heisenberg-type-Heisenberg-manifolds}
	The normalised Heisenberg manifold $\left(\Gamma^r \ba H_n,\m\right)$ with $\m=(h,g)$ is of Heisenberg type if and only if  $h\in\SPs_{2n}\left(g^{\sfrac 12}\Id\right)$ if and only if $d_k(h)=g^{-\sfrac 12}$ for all $1\leq k \leq n$. Accordingly, the set
	\begin{align*}
	\mathcal{M}_{n}^{HT}:=\bigcup_{r\in\mathcal{D}_n}\mathcal{M}_{n}^{r,HT}\,,& \quad \mathcal{M}_{n}^{r,HT} :=\left\{[(h,g)]\in \mathcal M^r_{n}\,|\,d_k(h)=g^{-\sfrac 12}, 1\le k\le n \right\}
	\end{align*}
	parametrises the isometry classes of normalised Heisenberg manifolds of Heisenberg type. Moreover, each $\mathcal{M}_{n}^{r,HT}$ is closed in $\mathcal{M}_{n}^{r}$.
\end{corollary}
\begin{proof}
	By the last proposition $\left(\Gamma^r \ba H_n,\m\right)$ is of Heisenberg type if and only if $d_k(h)=g^{-\sfrac 12}$ for all $1\leq k\leq n$. This, in turn, is the case if and only if $h\in\SPs_{2n}\left(g^{\sfrac 12}\Id\right)$. It follows from Theorem~\ref{thm:moduli-space-normalised-heisenberg-manifolds} that $\mathcal M_n^{HT}$ is a moduli space for normalised Heisenberg manifolds of Heisenberg type. 	
	 The set $\mathcal{M}_{n}^{r,HT}$ is closed in $\mathcal{M}_{n}^{r}$ because the $d_k$ are continuous.
\end{proof}

We will now study the compact sets of the moduli space $\mathcal{M}_n^r$.
\begin{notandrem}\label{rem:minkowski-fundamental-domain}
	\leavevmode
	\begin{enumerate}[(i)]
		\item \textit{Minkowski's fundamental domain} in $\SP_n$ is the domain
		\begin{multline*}
			\mathscr M_n = \big\{Y=\left(y_{i,j}\right)\in\SP_n\,|\,\forall k=1,\ldots,n: y_{k,k+1}\geq 0 \\ \text{ and } Y[a]\geq y_{k,k} \text{ for all } a\in\Z^{n} \text{ with } \gcd(a_k,\ldots,a_n)=1 \big\}\,.   
		\end{multline*}
		By \cite[CH.~IV, Thm.~1]{zbMATH00041535}, $\mathscr M_n$ is a connected and closed fundamental domain for $\SP_n / \GL(n;\R)$.
		\item \label{rem:enum:first-minimum} For $Y\in\SP_n$ we define
		\begin{align*}
			m(Y):=\inf\{Y[a]\,|\,a\in\Z^n\ba\{0\}\}\,.
		\end{align*}
		The value $m(Y)$ is called \textit{the first minimum of} $Y$. It is the squared norm of a shortest nonzero vector of a lattice with Gram matrix $Y$.
		Note that $m(Y)=y_{1,1}$ for $Y\in\mathscr M_n$ by the very definition of $\mathscr M_n$.

		\item   For $r\in\mathcal D_n$ and $Y\in\SP_{2n}$ we define
				\begin{equation*}
					m_r(Y):=\inf\{Y[\delta_r a]\,|\,a\in\Z^{2n}\ba\{0\}\}=m(Y[\delta_r])\,.
				\end{equation*}
				If $r=(1,\ldots,1)$, we abbreviate $m_r(Y)$ simply to $m(Y)$ which is in accordance with \eqref{rem:enum:first-minimum}. 
		\item The function $m:\SP_n\to(0,\infty)$ is constant on the orbits of the action of $\GL(n;\Z)$ on $\SP_n$ and we denote the induced function on $\SP_n/ \GL(n;\Z)$ by $m$, too.
		
		Similarly, the function $m_r:\SP_{2n}\to (0,\infty)$ is constant on the orbits of the action of  $G_r=\delta_r\GL(2n;\Z)\delta_r^{-1}$ on $\SP_{2n}$ and we denote the induced function on the quotient by $m_r$ as well.
	\end{enumerate}
\end{notandrem}

\begin{remark}\label{rem:minkowski-fundamental-domain-and-relatives}
	For $r\in\mathcal D_n$, define the map 
	\begin{align*}
		\Psi_r:\SP_{2n}&\to\SP_{2n}\\
			Y&\mapsto Y\left[\delta_r^{-1}\right]=\delta_r^{-1}Y\delta_r^{-1}\,.
	\end{align*}
	 Then $\Psi_r$ induces a map
	 \begin{align*}
		 \psi_r:\SP_{2n} / \GL(2n;\Z)&\to \SP_{2n} / G_r\\
		 [Y] &\mapsto [Y[\delta_r^{-1}]]=[\delta_r^{-1}Y\delta_r^{-1}]\,.
	 \end{align*}
	  The maps $\Psi_r$ and $\psi_r$ are diffeomorphisms and satisfy $\pi_r\circ\Psi_r = \psi_r\circ\pi$ where $\pi:\SP_{2n}\to\SP_{2n}/\GL(2n;\Z)$ and $\pi_r:\SP_{2n}\to\SP_{2n}/G_r$ are the canonical projections. It follows that
	  \begin{align}\label{eqn:definition-M_n,r-fundamental-domain}
	  	\mathscr M_{2n,r}:=\Psi_r(\mathscr M_{2n})
	  	\end{align} is a fundamental domain for the space $\SP_{2n} / G_r$. Note that $m_r(\Psi_r(Y)) = m(Y)$ and $m_r(\psi_r([Y])) = m([Y])$.
\end{remark}

The following classic theorem characterises the relatively compact subsets of $\mathscr M_n$ and is attributed to Hermite and Mahler by A.~Terras.
\begin{theorem}[Hermite-Mahler Compactness Theorem, see {\cite[CH.~4.4, Ex.~13]{zbMATH00041535}}]\label{thm:hermite-mahler}
	Any set $M\subset\mathscr M_n$ for which there are positive constants $C_0,C_1>0$ such that $m(Y)\geq C_0$ and $\det Y\leq C_1$ for all $Y\in M$ has compact closure in $\mathscr M_n$.
\end{theorem}

From this one easily deduces Mahler's compactness theorem in its original form, which we include for the sake of completeness.

\begin{corollary}[Selection Theorem of Mahler, cf. {\cite[Ch.\ 3 \S17 Theorem 2]{zbMATH03987367}}]\label{cor:Mahler-Selection-Theorem}\leavevmode\newline
	Let $M\subset\SP_n / \GL(n;\Z)$ be a set such that there are constants $C_0,C_1>0$ so that $m([Y])\geq C_0$ and $\det([Y])\leq C_1$ for all $[Y]\in M$. Then $M$ has compact closure in $\SP_n / \GL(n;\Z)$.
\end{corollary}
\begin{corollary}\label{cor:skewed-r-Mahler-Selection-Theorem}\mbox{}
	Let $n\in\N$, $r\in\mathcal D_n$ and $M\subset\SP_{2n} / G_r$ be a set such that there are constants $C_0,C_1>0$ so that $m_r([Y])\geq C_0$ and $\det([Y])\leq C_1$ for all $[Y]\in M$. Then $M$ has compact closure in $\SP_{2n} / G_r$.
\end{corollary}
\begin{proof}
	This follows from the last corollary via the diffeomorphism $\psi_r$ defined in Remark~\ref{rem:minkowski-fundamental-domain-and-relatives}.
\end{proof}
	
The following example shows that Corollary~\ref{cor:skewed-r-Mahler-Selection-Theorem} does not remain true if we replace $\SP_{2n}/G_r$ by $\SP_{2n}/\Pi_r$.
	
\begin{example}\label{ex:noncompact-set-in-moduli-space}
	Let $k\in\N_0$ and 
	\begin{equation*}
		Y_k:=\begin{pmatrix}	1 & k & 0 & 0 \\		k & k^2+1 & 0 & 0\\		0 & 0 & 1 & 0\\		0 & 0 & 0 & 1	\end{pmatrix} = \ltrans{\begin{pmatrix}		1 & k & 0 & 0 \\		0 & 1 & 0 & 0 \\		0 & 0 & 1 & 0 \\ 0 & 0 & 0 & 1 \\	\end{pmatrix}} \cdot \Id \cdot \begin{pmatrix}		1 & k & 0 & 0 \\		0 & 1 & 0 & 0 \\		0 & 0 & 1 & 0 \\ 0 & 0 & 0 & 1 \\	\end{pmatrix}\,.
	\end{equation*}
	Obviously, $\det(Y_k)=1$ for all $k\in\N_0$. Also, since $Y_k$ is in the same $\GL(4;\Z)$-orbit as $\Id$, $m(Y_k)=1$ for all $k\in\N_0$. One easily calculates
	\begin{align*}
		Y_k^{-1}\cdot J = \begin{pmatrix}	k^2+1 & -k & 0 & 0 \\	-k & 1 & 0 & 0 \\	0 & 0 & 1 & 0\\		0 & 0 & 0 & 1\\	\end{pmatrix}\cdot J = \begin{pmatrix}		0 & 0 & k^2+1 & -k \\		0 & 0 & -k & 1 \\		-1 & 0 & 0 & 0 \\	0 & -1 & 0 & 0 \\	\end{pmatrix}\,,
	\end{align*}
	so that
	\begin{align*}
		\left(Y_k^{-1}\cdot J \right)^2= \begin{pmatrix}	-k^2-1 & k & 0 & 0 \\	k & -1 & 0 & 0 \\	0 & 0 & -k^2-1 & k\\		0 & 0 & k & -1\\	\end{pmatrix}\,.
	\end{align*}
	The eigenvalues of $\left(Y_k^{-1}\cdot J \right)^2$ are thus the solutions of
	\[
		0=\left((X+k^2+1)(X+1)-k^2\right)^2=\left(X^2+(k^2+2)X+1\right)^2\,.
	\]
	It follows that
	\begin{align*}
		d_1(Y_k)&=2^{-\sfrac12}\sqrt{k^2+2-k\sqrt{k^2+4}}\,,\\
		d_2(Y_k)&=2^{-\sfrac12}\sqrt{k^2+2+k\sqrt{k^2+4}}\,.
	\end{align*}
	The sequence $d_2(Y_k)$ is monotonously and unboundedly increasing in $k$. Since $d_2$ is an invariant of the $\widetilde{\Sp}(4;\R)$-action (see Remarks~\ref{def:d_j-h-and-remarks}\eqref{def:d_j-h-and-remarks:enum:invariance-under-sp}), no two matrices of the family $\{Y_k\}_{k\in\N}$ are in the same $\widetilde{\Sp}(4;\Z)$-orbit. Note that the $d_j(Y)$, $j=1,2$, are continuous in $Y\in\SP_4$ (see Remarks~\ref{def:d_j-h-and-remarks}\eqref{def:d_j-h-and-remarks:enum:continuity}) and descend to continuous functions on $\SP_4/\widetilde{\Sp}(4;\Z)$. Therefore, no subsequence of $\{[Y_k]\}_{k\in\N}$ converges. In particular, boundedness of $\det(\cdot)$ and $m_{(1,1)}(\cdot)$ on a set $M\subset\SP_4/\Pi_{(1,1)}$ is not sufficient for $M$ to be relatively compact.
\end{example}

\begin{lemma}\label{lem:complete-representative-system-moduli-space-fund-group}	
	Let $n\in\N$, $r\in\mathcal{D}_n$ and let $U$ be a complete set of representatives for $G_r / \Pi_r$. Then, for any $C>0$ and any matrix norm $\|\cdot\|:M(2n;\R)\to[0,\infty)$ there are only finitely many $G\in U$ with $\|\ltrans G^{-1} J G^{-1}\|\leq C$.
\end{lemma}
\begin{proof}
	Firstly, for any $G,H\in U$ with $G\neq H$ we have $\ltrans G^{-1} J G^{-1} \neq \pm \ltrans H^{-1} J H^{-1}$. For if this were not the case, we would have $H^{-1}G=P\in\tSp(2n;\R)$, that is $[G]=[HP]=[H] \in G_r / \Pi_r $ which would be a contradiction. Secondly, if $G \in U$, then the entries of  $\ltrans G^{-1}JG^{-1}$ are elements of $\tfrac 1{r_n^2}\Z$. Since $M(2n;\R)$ is a finite dimensional vector space, all norms are equivalent and we can choose a particular one. Let $\|\cdot\|$ be the maximum norm $\|G\|=\max\{|G_{i,j}|\,|\,1\leq i,j\leq 2n\}$. Then $\|\ltrans G^{-1}JG^{-1}\|\geq r_n^{-2}$ and by the above, $\|\ltrans G^{-1}JG^{-1}-\ltrans H^{-1}JH^{-1}\|\geq r_n^{-2}$ for all $G,H\in U$ with $G\neq H$. The lemma's statement now follows from the fact that the closed norm ball $\{G\in M(2n;R)\,|\,\|G\|\leq C\}$ is compact.	
\end{proof}

\begin{notandrem}\label{nar:eigenvalue-functions}
	\leavevmode
	\begin{enumerate}[(i)]
		\item By the spectral theorem for symmetric matrices the eigenvalues of any matrix $Y\in\SP_n$ are real and positive. We introduce functions $\lambda_1,\ldots,\lambda_n:\SP_n\to (0,\infty)$ such that $\lambda_k(Y)$ is an eigenvalue of $Y$ and $\lambda_1\leq\lambda_2\leq\cdots\leq\lambda_n$. Note that the $\lambda_k$, $1\leq k\leq n$, are continuous functions (see, e.g., \cite[Theorem~1]{zbMATH03214567}).
		\item We furthermore define functions $s_1,\ldots,s_n:\GL(n;\R)\to (0,\infty)$ such that $s_k(G)$, $1\leq k\leq n$, are the \textit{singular values} of $G$ and $s_1 \geq s_2 \geq \cdots \geq s_n$. These functions satisfy $s_k(G)=\sqrt{\lambda_{n-k+1}(\ltrans GG)}$ which means in particular that they are continuous. Care has to be taken though as the $s_k$ are in general not invariant under conjugation. If, however, $\ltrans GG\sim A\in \SP_n$, then $s_k(G)$, $1\leq k\leq n$, are the square roots of the eigenvalues of $A$.
		\item \label{nar:eigenvalue-functions:enum:d_j-s_k-relation} We want to relate the functions $d_j:\SP_{2n}\to\R$, $1\leq j\leq n$, from Definition~\ref{def:d_j-h-and-remarks} to the singular values $s_k:\GL(2n;\R)\to (0,\infty)$, $1\leq k\leq 2n$. For any $G\in\GL(2n;\R)$ we have 
		\begin{align*}
			\ltrans{ \left(\ltrans G^{-1}JG^{-1}\right) }\left(\ltrans G^{-1}JG^{-1}\right)& = -\ltrans G^{-1}JG^{-1}\ltrans G^{-1}JG^{-1}\\
					& \sim -G^{-1}\ltrans G^{-1}JG^{-1}\ltrans G^{-1}J= -\left(\left(\ltrans GG\right)^{-1}J\right)^2\,,
		\end{align*}
		which implies 
		\[
			d_j(\ltrans GG)=s_{k}(\ltrans G^{-1}JG^{-1})
		\] for $k\in\{2n-2j+2,2n-2j+1\}$.
	\end{enumerate}
\end{notandrem}

The following lemma establishes the key matrix inequality required in our main theorem below.

\begin{lemma}\label{lem:moduli-space-special-inequality}
	The inequality
	\begin{equation*}
		d_n\left(\ltrans G G\right)\left(\lambda_{2n}(Y)\right)^{-1}\leq d_n(Y[G])
	\end{equation*}
	holds for all $Y\in\SP_{2n}$ and $G\in\GL(2n;\R)$.
\end{lemma}
\begin{proof}
	By \cite[p. 72, (III.20)]{zbMATH00967931} one has
	\begin{equation}\label{eqn:bhatia-singular-value-inequality}
		\prod\limits_{j=1}^{k}s_{i_j}(A)\prod\limits_{j=1}^{k}s_{2n-i_j+1}(B)\leq \prod\limits_{j=1}^{k}s_{j}(AB)
	\end{equation}
	for all $A,B\in M(2n;\R)$ and all $1\leq i_1<\cdots<i_k\leq 2n$. We choose $k=1$, $i_1=2n$, $A=Y^{-\sfrac 12}$ and $B=\ltrans G^{-1} J G^{-1} Y^{-\sfrac 12}$ and obtain
	\begin{equation*}
		s_{2n}\left(Y^{-\sfrac 12}\right)s_1\left(\ltrans G^{-1}JG^{-1}Y^{-\sfrac 12}\right)\leq s_1\left(Y^{-\sfrac 12}\ltrans G^{-1}JG^{-1}Y^{-\sfrac 12}\right)\,.
	\end{equation*}
	We apply \eqref{eqn:bhatia-singular-value-inequality} again to the second factor of the left hand side of this inequality, this time with $A=\ltrans  G^{-1}JG^{-1}$, $B=Y^{-\sfrac 12}$, $k=1$ and $i_1=1$, which yields
	\begin{equation}\label{eqn:d_n-inequality-0}
		\left(s_{2n}\left(Y^{-\sfrac 12}\right)\right)^2s_1\left(\ltrans G^{-1}JG^{-1}\right)\leq s_1\left(Y^{-\sfrac 12}\ltrans G^{-1}JG^{-1}Y^{-\sfrac 12}\right)\,.
	\end{equation}
	Now $\left(s_{2n}\left(Y^{-\sfrac 12}\right)\right)^2=s_{2n}\left(Y^{-1}\right)=\lambda_1\left(Y^{-1}\right)=\left(\lambda_{2n}(Y)\right)^{-1}$. Furthermore, one has \linebreak $s_1\left(\ltrans G^{-1}JG^{-1}\right) = d_n(\ltrans GG)$ by Remark~\ref{nar:eigenvalue-functions}\eqref{nar:eigenvalue-functions:enum:d_j-s_k-relation}. Together, this shows that the left hand side of \eqref{eqn:d_n-inequality-0} matches the left hand side of the inequality in the statement of the lemma. We have a look at the right hand side of \eqref{eqn:d_n-inequality-0}. Let $H:=Y^{\sfrac 12}G$. By Remark~\ref{nar:eigenvalue-functions}\eqref{nar:eigenvalue-functions:enum:d_j-s_k-relation} we have
	\begin{multline*}
			s_1\left(Y^{-\sfrac 12}\ltrans G^{-1}JG^{-1}Y^{-\sfrac 12}\right) = s_1\left(\ltrans H^{-1}JH^{-1}\right) = d_n\left(\ltrans H H\right)=d_n\left(\ltrans GY^{\sfrac 12}Y^{\sfrac 12}G\right)\\=d_n\left(Y[G]\right),
	\end{multline*}
	which finishes the proof of the stated inequality.
\end{proof}

\begin{theorem}\label{thm:compact-sets-in-moduli-space}
	Let $M\subset \SP_{2n} / \Pi_r$. Assume that there are positive constants $C_0,C_1$ and $C_2$ such that $m_r([Y])\geq C_0$, $\det([Y])\leq C_1$ and $d_n([Y])\leq C_2$ for all $[Y]\in M$. Then, $M$ has compact closure in $ \SP_{2n} / \Pi_r$.
\end{theorem}
\begin{proof}
We denote by $\pi_r:\SP_{2n} \to \SP_{2n}/G_r$, $p_r:\SP_{2n}\to\SP_{2n} / \Pi_r$ and $\eta_r:\SP_{2n}/\Pi_r\to \SP_{2n}/G_r$ the canonical projections. Note that $\eta_r\circ p_r = \pi_r$.\\
By Corollary~\ref{cor:skewed-r-Mahler-Selection-Theorem}, the set $\eta_r(M)$ is precompact. There is thus a precompact set $K\subset \mathscr M_{2n,r}$, where $\mathscr M_{2n,r}$ is the fundamental domain for $\SP_{2n}/\Pi_r$ defined by \eqref{eqn:definition-M_n,r-fundamental-domain}, with $\pi_r(K)=\eta_r(M)$. Consequently, one has
\begin{align}
	M\subseteq \bigcup_{G\in U} p_r\left(\ltrans GKG\right)\,,
\end{align}
where $U$ is a full set of representatives of $G_r/\Pi_r$.
The function $\lambda_{2n}$ is continuous and therefore takes its maximum $\overline{M}>0$ on the closure $\overline K$ of $K$. By Lemma~\ref{lem:moduli-space-special-inequality} one has 
\begin{align}\label{eqn:dn-larger-and-larger}
d_n\left(\ltrans GYG\right)\geq \left(\lambda_{2n}(Y)\right)^{-1}d_n\left(\ltrans GG\right)\geq \overline{M}^{-1}d_n\left(\ltrans GG\right)\, \text{ for all } Y\in K, G\in U\,.
\end{align}
 Let $\|\cdot\|_2$ be the spectral norm, that is, $\|G\|_2=s_1(G)$. Also, recall that $d_n\left(\ltrans GG\right)=s_1\left(\ltrans G^{-1}JG^{-1}\right)$ by Remark~\ref{nar:eigenvalue-functions}\eqref{nar:eigenvalue-functions:enum:d_j-s_k-relation}. Then, by Lemma~\ref{lem:complete-representative-system-moduli-space-fund-group} there are only finitely many $G\in U$ with $d_n\left(\ltrans GG\right)=s_1\left(\ltrans G^{-1}JG^{-1}\right)=\|\ltrans G^{-1}JG^{-1}\|_2\leq C_2{\overline M}$. Let $G_1,\ldots,G_N$ be those $G\in U$. Because of inequality~\eqref{eqn:dn-larger-and-larger} we have
 \begin{align*}
 	d_n\left(\ltrans G Y G\right) >C_2 \text{ for all } Y\in K \text{ and all } G\in U\ba\{G_1,\ldots,G_N\}\,,
 \end{align*}
  which in turn implies, by the assumption on $d_{n|M}$, that
\begin{align*}
	M\subseteq \bigcup_{j=1}^N p_r\left(\ltrans G_j K G_j\right)\,.
\end{align*}
The right hand side is a finite union of precompact sets and hence precompact. The set $M$ has thus compact closure, as claimed.
\end{proof}

We can now characterise the relatively compact sets of the moduli spaces $\mathcal M_n^r$ and $\mathcal M_n^{r,HT}$.

\begin{corollary}\label{cor:compact-sets-in-M_n,r}
	Let $n\in\N$, $r\in\mathcal D_n$ and $M\subset \mathcal M^r_{n}$. Assume that there are positive constants $C_0, C_1, C_2>0$ and a compact interval $I\subset (0,\infty)$ such that $g\in I$, $m_r(h)\geq C_0$, $\det(h)\leq C_1$ and $d_n(h)\leq C_2$ for all $[(h,g)]\in M$. Then $M$ has compact closure.
\end{corollary}

\begin{corollary}\label{cor:compact-sets-of-heisenberg-type-moduli}
	Let $n\in\N$ and $r\in\mathcal D_n$. Then any set $M\subset \mathcal M_{n}^{r,HT}$ for which there exists a constant $C_0> 0$ and a compact interval $I\subset (0,\infty)$ such that $m_r([h])\geq C_0$ and $g\in I$ for all $[(h,g)]\in M$ is relatively compact.
\end{corollary}
\begin{proof}
	Let $M\subset \mathcal M_{n}^{r,HT}$ be a set for which there exist constants as stated. Any $[(h,g)]\in\mathcal M_{n}^{r,HT}$ satisfies $d_k(h)=g^{-\sfrac 12}$ for all $1\le k\le n$ and hence $\det h = det(J^{-1}h)=g^n$. It follows from the last corollary that $M$ has compact closure in $\mathcal M_n^{r}$.	By Corollary~\ref{cor:Heisenberg-type-Heisenberg-manifolds}, $\mathcal M_{n}^{r,HT}$ is a closed subspace of $\mathcal M_{n}^r$. Hence, the closure of $M$ is still contained in $\mathcal M_{n}^{r,HT}$.
\end{proof}

\begin{remark}\label{rem:sectional-curvatures}
	Let $(\Gamma^r\ba H_n,\m)$ be a normalised Heisenberg manifold with $\m=(h,g)$. By \eqref{eqn:sec-curvatures} and \eqref{eqn:Kaplan-j-matrix} we have the following inequality for the sectional curvatures $K$,
	\[
		K \leq g^{-1}d_n^2(h)\,.
	\]
	 Hence one can equivalently replace the uniform bound on $d_n$ in Corollary \ref{cor:compact-sets-in-M_n,r} by a uniform upper bound on the sectional curvatures. 
\end{remark}


\begin{thebibliography}{Zed65}
\bibitem[Bha96]{zbMATH00967931}
	{\sc R. {Bhatia}}
	{\it {Matrix analysis.}}
	Springer New York, 1996.
\bibitem[Bol18]{Bol18}
	{\sc S. {Boldt}}
	{\it The height of compact nonsingular Heisenberg-like Nilmanifolds.}
	Dessertation, Doi: 10.18452/18924, 2018.
\bibitem[Ebe94]{Ebe94}
	{\sc P. {Eberlein}}
	{\it Geometry of {$2$}-step nilpotent groups with a left invariant metric.}
	Ann. Sci. \'{E}cole Norm. Sup. (4) 27, no.~5, 1994, 611--660.
\bibitem[GL87]{zbMATH03987367}
	{\sc P.M. {Gruber} and C.G. {Lekkerkerker}}
	{\it {Geometry of numbers. 2nd ed.}}
	North Holland, 1987.
\bibitem[GW86]{zbMATH03967358}
	{\sc C. S. {Gordon} and E. N. {Wilson}}
	{{The spectrum of the Laplacian on Riemannian Heisenberg manifolds.}}
	{\it {Mich. Math. J.}}, 33, 1986, 253--271.
\bibitem[Tas20]{Tas20}
	{\sc K. {Tashiro}}
	{\it Precompactness theorem for compact Heisenberg manifolds with sub-Riemannian metrics and the Gromov-Hausdorff topology.}
	arXiv preprint. arXiv:2004.09407, 2020.
\bibitem[Ter88]{zbMATH00041535}
	{\sc A. {Terras}}
	{\it {Harmonic analysis on symmetric spaces and applications. II.}}
	Springer New York, 1988.
\bibitem[Zed65]{zbMATH03214567}
	{\sc M. {Zedek}}
	{{Continuity and location of zeros of linear combinations of polynomials.}}
	{\it {Proc. Am. Math. Soc.}}, 16, 1965, 78--84.
\end{thebibliography}
\end{document}